\documentclass[11pt]{amsart}

\usepackage{amsfonts}
\usepackage{mathtools}
\usepackage{amssymb}
\usepackage{graphicx}
\usepackage{amsmath}
\usepackage{latexsym}
\usepackage{amscd}
\usepackage{xypic}
\usepackage{mathrsfs}
\usepackage{enumitem} 
\usepackage{braket}
\usepackage[margin=1.5in]{geometry}
\usepackage{bm}
\usepackage{color}
\usepackage{stmaryrd}
\usepackage{marginnote}
\usepackage[backend=biber, style=alphabetic]{biblatex}
\usepackage{imakeidx}
\makeindex[title=Index of Notation]

\usepackage[colorinlistoftodos,prependcaption]{todonotes}
\usepackage{hyperref}

\allowdisplaybreaks 

\numberwithin{equation}{section}
\newtheorem{theorem}{Theorem}[section]
\newtheorem{lemma}[theorem]{Lemma}
\newtheorem{corollary}[theorem]{Corollary}

\newtheorem{proposition}[theorem]{Proposition}

\theoremstyle{remark}
\newtheorem{remark}[theorem]{Remark}
\newtheorem{claim}[theorem]{Claim}

\newtheorem*{theorem*}{Theorem}

\theoremstyle{definition}
\newtheorem{definition}[theorem]{Definition}

\theoremstyle{definition}

\newcommand{\eps}{\varepsilon}

\DeclareMathOperator{\bee}{\mathcal{B}_{\eps}}
\DeclareMathOperator{\bei}{\mathcal{B}_{\eps_{i}}}
\DeclareMathOperator{\het}{\bf{H}}

\title{A strong min-max property for level sets of Phase Transitions}

\author{\'Erico Melo Silva} 
\address{Department of Mathematics, Princeton University, Washington Road, Princeton, NJ, 08540, USA}
\email{easilva@princeton.edu}
\date{\today}

\begin{document}

\begin{abstract}
We show that certain functions whose nodal sets lie near a fixed nondegenerate minimal hypersurface satisfy a strong min-max principle for the Allen--Cahn energy which is analogous to the strong min-max principle for non-degenerate minimal hypersurfaces first proved by Brian White \cite{white1994strong}.
\end{abstract}

\maketitle

\section{Introduction}\label{IntroSec}

Given a Riemannian manifold $(M^{n}, g)$ a minimal hypersurface $\Sigma^{n -1} \subset M$ is one which arises as a \textit{critical point} for the area functional. The variational properties of minimal hypersurfaces have been a rich source of study for geometric analysis. In particular, the properties of the \textit{stability operator}, $L_{\Sigma}$, of a minimal hypersurface often determine important aspects of geometric variational problems. 

In \cite{white1994strong}, B. White proved that any minimal submanifold which is nondegenerate (in the sense that it admits no non-zero Jacobi fields, i.e. the equation $L_{\Sigma} f = 0$ has no nontrivial solutions), and finite Morse index is the solution of a strong min-max problem.

The Allen--Cahn energy for functions $u \in W^{1, 2}(M)$ is given by

\begin{equation}\label{ACenergy}
    E_{\eps}(u) \doteq \int_{M} \frac{\eps|\nabla u|^{2}}{2} + \frac{W(u)}{\eps},
\end{equation}

where $W(t) = \frac{(1 - t^{2})^{2}}{4}$ is the standard double--well potential. The Euler--Lagrange equation of \ref{ACenergy} is known as the \textit{Allen--Cahn equation} and is given by

\begin{equation}\label{ACeq}
    \eps^{2} \Delta_{g} u = W'(u).
\end{equation}

The solutions to \ref{ACeq} have a well known correspondence with minimal hypersurfaces, beginning with Modica--Mortola's work (see e.g., \cite{modica1985gradient}, \cite{mortola1977esempio}) on the $\Gamma$--convergence of the energies \ref{ACenergy} to the perimeter functional acting on sets of finite perimeter. A rich body of work has emerged in recent years utilizing min--max properties of the Allen--Cahn equation \ref{ACeq} to provide variational constructions of minimal hypersurface. Of particular note in this direction is the work of Guaraco \cite{guaraco2018min}, Guaraco--Gaspar \cite{gaspar2018allen}, Chodosh--Mantoulidis \cite{chodosh2020minimal}, and many others.

In the other direction, constructing solutions to the Allen--Cahn equation given a fixed minimal hypersurface, we have the following result of Pacard--Ritoré \cite{pacard2003constant}.

\begin{theorem}[Theorem 1.1 \cite{pacard2003constant}]\label{PacRit}
Assume that $(M, g)$ is an $n$-dimensional closed Riemannian manifold and $\Sigma^{n-1} \subseteq M^n$ is a $2$-sided, nondegenerate minimal hypersurface. Then there exists $\eps_0 > 0$ such that $\forall \eps < \eps_0$, there exist solutions, $u_{\eps}$, to equation \eqref{ACeq} such that $u_{\eps}$ converges to $+1$ (resp. -1) on compact subsets of $(M^+)^o$ (resp. $(M^{-})^o$) and 
\[
E_{\eps}(u_{\eps}) \xrightarrow{\eps \to 0} 2\sigma_{0} \mathcal{A}(\Sigma)
\]
where $\mathcal{A}(\Sigma)$ the $(n-1)$--dimensional area of $\Sigma$ and $\Sigma$ is a constant which depends on the choice of potential in \ref{ACenergy}.
\end{theorem}

In \cite{marxkuo2023geometric} the author, together with J. Marx--Kuo, gave a variant of the Allen--Cahn energy \ref{ACenergy} suited to the study of the geometry of level sets of solutions. In this paper, we prove that our variant of the Allen--Cahn energy satisfies a strong min--max property analogous to that proved in White \cite{white1994strong}, provided that we restrict to surfaces which lie over the minimal hypersurface as small normal graphs. As a geometric application of our result, we give a new, purely variational proof of \ref{PacRit}.

\begin{remark}
    The first variational proof of \ref{PacRit} was given recently by Pigati--De Philippis \cite{dephilippis2022nondegenerate}. In their work, the candidate solutions to the Allen--Cahn equation are constructed using gradient flow methods. In the present work, our construction relies on the variational properties of the associated surface energy first defined in \cite{marxkuo2023geometric}, as well as a variant of the classical mountain pass theorem for solutions to elliptic semilinear equations.
\end{remark}

\subsection{Acknowledgements.}

The author is indebted to his advisor, Fernando Cod\'a Marques, for his patience and essential conversations regarding the completion of the present work. The author would also like to thank Jared Marx--Kuo for innumerable mathematical conversations. 

\section{Preliminaries and Main Results}\label{NotandPreSec}

\subsection{Notation and conventions.} \label{notandconsec}

\begin{itemize}
    \item $(M^{n}, g)$ always denotes an orientable, closed Riemannian manifold.
    \item Unless specified otherwise $\Sigma \subset M$ denotes an embedded, closed hypersurface, i.e. a submanifold of $M$ of codimension $1$.
    \item $B^{k}$ is the unit ball in $\bf{R}^{k}$.
    \item If $E \subset M$ is a Borel set, then $\mathcal{C}(E)$ denotes the space of sets of finite perimeter (Caccioppoli sets) $\Omega \subset E$.
    \item We say a sequence of Caccioppoli sets $\Omega_{n} \to \Omega$ if the measure of the symmetric difference tends to $0$, (i.e. $|\Omega_{n} \Delta \Omega| \to 0$.
    \item We denote by $C^{k, \alpha}_{\eps}(M)$ the Banach space of $C^{k, \alpha}$ functions on $M$ equipped with the $\eps$-weighted H\"older norms

    \begin{align}
        \|f\|_{C^{k, \alpha}_{\eps}} \doteq \sum_{|\beta| \leq k} \eps^{|\beta|}\|D^{\beta} f\|_{C^{0}(M)} + \eps^{k + \alpha}[D^{k}f]_{\alpha},
    \end{align} 

    where $\beta$ are multiindices and $[\cdot]_{\alpha}$ is the usual $\alpha$--H\"older seminorm.

    \item We denote by $W^{1, 2}_{\eps}(M)$ a Banach space we call the $\eps$-weighted Sobolev space, with norm

    \begin{align}
        \|f\|_{W^{1, 2}_{\eps}(M)}^{2} \doteq \eps\|f\|_{L^{2}(M)} + \eps^{3}\|\nabla f\|_{L^{2}(M)}.
    \end{align}

    Similarly we denote by $W^{1, 2}_{0, \eps}(M)$ those functions in $W^{1, 2}_{\eps}(M)$ which have trace $0$.
    \item The properties of the one-dimensional solution $\het(z) = \tanh(t/\sqrt{2})$ to the Allen--Cahn equation $\partial_{z}^{2}u - W'(u) = 0$. We define four functions as the unique solutions to auxiliary ODEs: $\omega, \rho, \tau, \kappa$,

    \begin{align} \label{AuxODEs}
        \omega'' - W''(\het)\omega &= \partial_{z}\het',\\
        \rho'' - W''(\het)\rho &= \omega',\\
        \tau'' - W''(\het)\tau &= z\het',\\
        \kappa'' - W''(\het)\kappa &= \het \omega,
    \end{align}

    all with the boundary conditions $f(0) = 0$, $\lim_{z \to \infty} f(z) = 0$.

    \item We define constants

    \begin{align*}
        \sigma_{0} \doteq \int_{0}^{\infty} \partial_{z}\het(z) dz = \frac{\sqrt{2}}{{3}},\\
        \sigma \doteq \het'(0) = \sqrt{2}^{-1}. 
    \end{align*}

    \item We write $\ell_{\eps, u} \doteq \Delta_{g} - W''(u)$ for the linearized Allen--Cahn operator at a solution $u$.

    \item For any function $f: \bf{R} \to \bf{R}$ we define $f_{\eps}(z) \doteq f(z/\eps)$.

    \item We define a cut-off of the heteroclinic, $\overline{\het}$, as follows. Let $\ell > 5$ $\chi_{\eps} : \bf{R} \to \bf{R}$ be a cutoff function.

    \begin{align*}
        \begin{cases}
            \chi_{\eps}(z) = 1, &|z| \in [0, -\ell \log(\eps)],\\
            \chi_{\eps}(z) = 0, &|z| \not\in [0, -2\ell \log(\eps)].
        \end{cases}
    \end{align*}

    Then

    \begin{align*}
        \overline{\het}(z) \doteq \chi_{\eps}(z)\het(z) + (1 - \chi_{\eps}(z)).
    \end{align*}

    Similarly, we can define the cut-off functions $\overline{\omega}$, $\overline{\rho}$, $\overline{\tau}$ and $\overline{\kappa}$, where these each cut-off to $0$ instead of to $1$, as the corresponding functions decay to $0$ at infinity. We note the cut-offs satisfy

    \begin{align*}
        \left(\frac{d^{2}}{dz^{2}} - W''(\overline{\het})\right)\overline{\het}' = E,\\
        \left(\frac{d^{2}}{dz^{2}} - W''(\overline{\het})\right)\overline{\omega}' = E + \overline{\het}',
    \end{align*}

    with $\|E\|_{W^{1, k}} \leq C\eps^{\ell}$.

    \item If $\Sigma$ is a sufficiently smooth minimal hypersurface, we denote by $H_{\Sigma}$ its mean curvature, and denote by $A_{\Sigma}$ its second fundamental form. 
    
    \item We denote by $\nu$ the outward pointing unit normal to $\Sigma$.

    \item We write by $N(\eta)$ a tubular neighborhood of $\Sigma$ of height $\eta$. Provided $\eta$ is sufficiently small, we may find a parametrization $\Sigma \times (-\eta, \eta) \to N(\eta)$, by

    \[
    (s, z) \mapsto \exp_{s}(z \nu(s)).
    \]

    We call this parametrization \textit{Fermi coordinates} in the tubular neighborhood $N(\eta)$.

    \item If $f \in C^{k, \alpha}(\Sigma)$ and $f$ is sufficently small, we denote by $\Gamma(f)$ the \textit{normal graph} of $f$ over $\Sigma$. That is, in Fermi coordinates,

    \[
    \Gamma(f) \doteq \{x \in M | x = \exp_{s}(f(s)\nu(s)\}.
    \]

    \item For the energies $E_{\eps}$, $\bee$ we denote by $\delta E_{\eps}, \delta^{2} E_{\eps}$, $\delta \bee,$ $\delta^{2} \bee$ their first and second variations.
\end{itemize}

In \cite{brezis1986remarks} Brezis-Oswald considered the solvability of the Dirichlet problem for the equation \ref{ACeq}.

\begin{theorem}[\cite{brezis1986remarks} Theorem $1$]\label{BOTheorem}
Suppose $\Omega \subset M$ is an open domain with smooth boundary. If $\varepsilon < \lambda_{1}(\Omega)^{-\frac{1}{2}}$, then

\begin{equation}\label{DirProb}
\begin{cases}
\Delta_{g} u = \frac{W'(u)}{\varepsilon^{2}} &\text{ in } \Omega,\\
u > 0 & \text{ in } \Omega,\\
u = 0 & \text{ on } \partial \Omega,
\end{cases}
\end{equation}
has a unique solution. Moreover, the solution minimizes the energy \ref{ACenergy} over all functions in $W^{1, 2}_{0}(\Omega)$.
\end{theorem}

It follows from \ref{BOTheorem} that for a separating hypersurface $\Sigma \subset M$ along which $M$ splits into a disjoint union $M = M^{+} \sqcup_{\Sigma} M^{-}$, there exists an $\eps_{0}$ sufficiently small so that for all $\eps < \eps_{0}$ the problem \ref{DirProb} is uniquely solvable in both $M^{+}$ and $M^{-}$. This observation led the author, together with J. Marx--Kuo in \cite{marxkuo2023geometric} to define the \textit{balanced energy} of $\Sigma$.

\begin{definition}
    Let $\Sigma \subset M$ be a $2$--sided, separating hypersurface splitting $M$ into two disjoint open sets $M^{+}$ and $M^{-}$. Suppose $\eps_{0}$ is sufficiently small that the problem \ref{DirProb} admits a solution for both $M^{+}$ and $M^{-}$. Call these solutions $u^{+}_{\Sigma, \eps}$ and $u^{-}_{\Sigma, \eps}$, respectively. Then we define

    \begin{align}\label{BEdef}
        \bee(\Sigma) \doteq \int_{M^{+}} \frac{\eps|\nabla u^{+}_{\Sigma, \eps}|^{2}}{2} + \frac{W(u^{+}_{\Sigma, \eps})}{\eps} + \int_{M^{-}} \frac{\eps|\nabla u^{-}_{\Sigma, \eps}|^{2}}{2} + \frac{W(u^{-}_{\Sigma, \eps})}{\eps}
    \end{align}
\end{definition}

When $u^{+}_{\Sigma, \eps}$ and $u^{-}_{\Sigma, \eps}$ are simultaneously defined, we write

\[
u_{\Sigma, \eps} \doteq \begin{cases}
    u^{+}_{\Sigma, \eps} &\text{in } M^{+},\\
    u^{-}_{\Sigma, \eps} &\text{in } M^{-},\\
    0 & \text{on } \Sigma,
\end{cases}
\]

and we call $u_{\Sigma, \eps}$ the \textit{broken} $\eps$--\textit{phase transition associated to} $\Sigma$. If $\Sigma$ and $\eps$ are understood, we omit writing them. 

\subsection{Preliminary results}\label{prelimSec}

Suppose $\Sigma \subset M$ is a smooth, closed, embedded minimal hypersurface. Suppose further that $\Sigma$ is nondegenerate and has Morse index $k$. In \cite{white1994strong} B. White proved the following theorem.

\begin{theorem}[\cite{white1994strong}, Theorem $4$] \label{WhiteTheorem}
With $\Sigma$ as above, there exists a tubular neighborhood $U$ of $\Sigma$, and a smooth $k$--parameter family of hypersurfaces homologous to $\Sigma$, $\{\Sigma_{v}\}_{v \in B^{k}}$ so that the following hold.

\begin{enumerate}
    \item $\Sigma_{0} = \Sigma$.

    \item $\sup_{v \in B^{k}} \mathcal{A}(\Sigma_{v}) = \mathcal{A}(\Sigma)$.

    \item If $\tilde{\Sigma}_{v}, v \in B^{k}$ is any other smooth $k$--parameter family of smooth, closed, embedded hypersurfaces in $U$, homologous to $\Sigma$, and $\tilde{\Sigma}_{v} = \Sigma_{v}$ for all $v \in \partial B^{k}$, then

    \begin{equation}
        \sup_{v \in B^{k}} \mathcal{A}(\tilde{\Sigma}_{v}) \geq \mathcal{A}(\Sigma),
    \end{equation}

    with equality if and only if $\tilde{\Sigma}_{v} = \Sigma$ for some $v \in B^{k}$.
\end{enumerate}
\end{theorem}

\begin{remark}
    The $k$-parameter family in Theorem \ref{WhiteTheorem} is given as follows. Let $u_{1}, \dots, u_{k}$ be the first $k$-eigenfunctions of the stability operator for $\Sigma$. For $v = (v_{1}, \dots, v_{k}) \in B^{k}_{\rho},$ $\Sigma_{v} = \Gamma(v_{1}u_{1} + \dots v_{k}u_{k})$. Here $\rho$ is taken sufficiently small that $v_{1}u_{1} + \dots v_{k}u_{k}$ lie in $U$.
\end{remark}

We call a neighborhood $V$ of $0$ in $C^{k,\alpha}(\Sigma)$ \textit{admissible} if, for any $f \in V$, $\Gamma(f) \subset U$, where $U$ is the tubular neighborhood given in Theorem \ref{WhiteTheorem}.

In \cite{marxkuo2023dirichlettoneumann}, J. Marx--Kuo computed an expansion of the solutions to \ref{DirProb} in $\eps$ to derive an asymptotic formula for solutions to the Dirichlet problem which makes clear the geometric dependency of solutions on the boundary of the domain. 

\begin{theorem}[\cite{marxkuo2023dirichlettoneumann}, Theorem 1.5]
    Let $u: \Omega \to \bf{R}$ be the solution to \ref{DirProb} afforded by \ref{BOTheorem}, where $\Sigma = \partial \Omega$. Then

    \begin{align}\label{uExp}
        u(s, z) &= \overline{\het}_{\eps}(z) + \eps H_{\Sigma}(s)\overline{\omega}_{\eps}(z)\\
        & \quad \quad + \eps^{2} \bigg((\textnormal{Ric}(\nu(s), \nu(s)) + |A_{\Sigma}(s)|^{2}) \overline{\tau}_{\eps}(z) \nonumber \\
        & \quad \quad + H_{\Sigma}^{2}(s)\bigg( \overline{\rho}_{\eps}(z) + \frac{1}{2}\overline{\kappa}_{\eps}(z) \bigg) \bigg) + \phi_{0, \eps}, \nonumber
    \end{align}

    where the remainder term $\phi_{0, \eps}$ satisfies the estimate,

    \begin{align}\label{rem1est}
    \|\phi_{0, \eps}\|_{C^{2, \alpha}_{\eps}(\Omega)} \leq C(\Sigma)\eps^{3}.
    \end{align}
\end{theorem}

In \cite{marxkuo2023geometric}, it was shown that the solutions $u_{\Sigma}$ depend in a smooth way on the hypersurface $\Sigma$. This allows one to take the first and second variations of the energies $\bee$. We collect the following results.

\begin{theorem}[\cite{marxkuo2023geometric}, Corollary 2.1]
If $\Sigma_{t}$ is a smooth $1$-parameter family of closed, $2$--sided, separating hypersurfaces for which $\bee$ is defined, and the variational vector field $X$ is given by $f \nu$ for some $f \in C^{k, \alpha}(\Sigma_{0})$, then

\begin{align}\label{firstVarest0}
    \frac{d}{dt}\bigg\vert_{t = 0} \bee(\Sigma_{t}) = \frac{\eps}{2} \int_{\Sigma_{0}} f((u^{+}_{\nu})^{2} - (u^{-}_{\nu})^{2}) = - 2\sigma_{0}\int_{\Sigma_{0}} H_{\Sigma_{0}}f d\mathcal{H}^{n-1} + E(f),
\end{align}

where $E(f) \leq C(\Sigma_{0})\eps \|f\|_{C^{k, \alpha}}(\Sigma)$.
\end{theorem}

Since $u_{t} \doteq u_{\Sigma_{t}, \eps}$ is a well defined family of functions in $W^{1, 2}(M)$ smooth in $t$, we denote by $\dot{u}_{t}$ its time derivative, which can be considered as a function in $W^{1,2}(M)$ which, following \cite{marxkuo2023geometric}, satisfies the boundary value problem

\begin{align}
    \begin{cases}
        \ell_{\eps, u_{t}} \dot{u}_{t} = 0 & \text{in } M^{+}_{t},\\
        \dot{u}_{t} = - f \frac{\partial u_{t}}{\partial \nu_{t}} & \text{on } \Sigma_{t}.
    \end{cases}
\end{align}

We have the following decomposition for the function $\dot{u}_{t}$.

\begin{proposition}{\cite{marxkuo2023geometric}, Lemma 9.1}
In fermi coordinates adapted to $\Sigma_{t}$, we may write

\begin{equation}\label{dotuExp}
    \dot{u}_{t}^{\pm}(s, z) = -f(s) \frac{\partial u_{t}^{\pm}}{\partial \nu_{t}} \dot{\overline{\het}}_{\eps}(z) + \phi_{1, \eps}^{\pm},
\end{equation}

where $\|\phi_{1, \eps}^{\pm}\|_{W^{1, 2}_{0, \eps}(M^{\pm})} \leq C(\Sigma)\eps\|f\|_{W^{1, 2}(\Sigma)}$.
\end{proposition}

For the second variation, the following was shown.

\begin{theorem}[\cite{marxkuo2023geometric}, Theorem 2.6]
    Let $\Sigma$ be a critical point for $\bee$. Then if $\Sigma_{t}$ is any normal variation with variational vector field given by $f \nu$ as above, then

    \begin{align}\label{secVarEstCrit}
        \frac{d^{2}}{dt^{2}}\bigg\vert_{t = 0} \bee(\Sigma_{t}) &= \eps\int_{\Sigma} f u_{\nu}(\dot{u}^{+}_{\nu} - \dot{u}^{-}_{\nu}),\\
        &=2\sigma_{0}\int |\nabla f|^{2} - (\textnormal{Ric}(\nu, \nu) + |A_{\Sigma}|^{2})f^{2}d\mathcal{H}^{n -1} + \tilde{E}(f), \nonumber
    \end{align}

    where $E(f) \leq C(\Sigma) \eps^{\frac{1}{2}} \|f\|_{H^{1}(\Sigma)}.$ 
    
\end{theorem}

\subsection{Main results.} \label{MainSec}

We now state the main results of this paper.

\begin{theorem}\label{min max}
    Let $\Sigma \subset (M, g)$ be a smooth, closed, embedded, $2$-sided, separating minimal hypersurface with Morse index equal to $k$ and nullity equal to $0$. 
    
     Let $\delta > 0$. There is an $\eps_{0} > 0$ so that for all $\eps < \eps_{0}$ the following holds. Write $\{\Sigma_{v} \vert v \in B^{k}\}$ for the canonical family given by \ref{WhiteTheorem} and $U$ an admissible neighborhood in $C^{k, \alpha}(\Sigma)$, and for $r > 0$ write $\mathcal{P}_{r} \doteq \{\gamma \in C^{0}(B^{k}, U) \vert \gamma \vert_{\partial B^{k}}(v) = \Sigma_{rv}\}$.

    \[
    \inf_{\gamma \in \mathcal{P}_{r}} \sup_{v \in B^{k}} \bee(\gamma(v)) \geq \bee(\Sigma) - \delta.
    \]
\end{theorem}

\begin{remark}
    Effectively, Theorem \ref{min max} says that any Allen--Cahn solution produced as a limit of broken phase transitions over surfaces which are small normal graphs over $\Sigma$ must have Allen--Cahn energy which is controlled by the area of $\Sigma$ from below. We note that the $\Gamma$-convergence result of \cite{marxkuo2023geometric} with \ref{min max} together imply the continuity of the min max value in $\eps$.
\end{remark}

Arguing as in the proof of the classical mountain pass principle for solutions to semilinear PDEs, we can show a mountain pass principle for the $\bee$ energies. We will state the following theorem more precisely as Theorem \ref{MPTheoremk}. Let $\mathcal{U}$ be a sufficently small neighborhood of $0$ in $C^{k, \alpha}(\Sigma)$ (here we take $k > 5$) so that the normal graph of any member of $\mathcal{U}$ is contained in the tubular neighborhood $U$ given by \ref{min max}.

\begin{theorem}\label{MPmain}
    Consider a family of normal graphs parametrized by the boundary of the unit $k$-ball $\{\Sigma_{v}\}_{v \in \partial B^{k}}$. Define $\mathcal{P}^{L}, c$ and $d$ as follows.

    \begin{align}
        \mathcal{P} &\doteq \{p \in C(B^{k}, \mathcal{U})| \text{ for every }v \in \partial B^{k}, \Gamma(p(v)) = \Sigma_{v}\},\\
        c &\doteq \max_{v \in \partial B^{k}} \bee(\Sigma_{v}),\\
        d &\doteq \inf_{p \in \mathcal{P}}\sup_{v \in B^{k}}\bee(\Gamma(p(v))).
    \end{align}

    If $d > c$, then there exists a smooth phase transition $u$ with $E_{\eps}(u) = d$, and moreover the level set $u^{-1}(0) \subset U$.
\end{theorem}

By combining \ref{WhiteTheorem} and \ref{MPmain}, we can then show

\begin{corollary}\label{PacRitVar}
If $\Sigma$ is a smooth, closed, embedded, separating, nondegenerate minimal hypersurface in a closed Riemannian manifold $M$, then there exists a sequence of phase transitions $u_{i}$ solving \ref{ACeq} for a sequence $\eps_{i} \to 0$ as $i \to \infty$ with

\[
E_{\eps_{i}}(u_{i}) \to 2\sigma_{0} \mathcal{A}(\Sigma).
\]

Moreover, $u_{i}^{-1}(0)$ Hausdorff-converges to $\Sigma$ as $i \to \infty$.
\end{corollary}

\section{First and second variations of energy}\label{12Var sec}

The goal of this section is to compute the first and second variations of $\bee$ at an arbitrary point, and to prove a result analogous to \ref{secVarEstCrit} for sufficiently small normal graphs over a minimal hypersurface.

Let $F_{t}: M \to M$ be a $1$-parameter family of diffeomorphisms and set,

\begin{itemize}
    \item $F_{0} = \text{id}_{M}$,
    \item $\partial_{t} F_{t} \doteq X_{t}$. We write $X \doteq X_{0}$ and call $X$ the variational vector field of $F_{t}$.
    \item $Z_{t} \doteq \nabla_{\partial_{t}} X_{t}$. We write $Z \doteq Z_{0}$ and call $Z_{t}$ the acceleration vector field for $F_{t}$.
\end{itemize}

\begin{remark}
    We can take $X = f\nu$ for $f \in C^{\infty}(\Sigma)$ and we can take $Z_{t} \equiv 0$ for all sufficiently small $t$ by assuming $F_{t}$ is of the form

    \begin{equation}\label{goodVar}
    F_{t}(s, z) = \exp_{s}(tf(s)\chi(z)\nu(s)),
    \end{equation}

    where $\chi(z)$ is a suitably chosen cutoff function which is identically $1$ in a neighborhood of $z = 0$.
\end{remark}

Write $M^{\pm}_{t} \doteq F_{t}(M^{\pm})$ and $\Sigma_{t} \doteq F_{t}(\Sigma)$. For small $t$, we may compute the first and second variations of $\bee(\Sigma_{t})$. For the sake of concision we omit the subscript $t$.

\begin{proposition}\label{FirSecVar}
    Let $u^{\pm}, \dot{u}^{\pm}$ be as above. Then,

    \begin{align} \label{firvar}
        \frac{d}{dt} \bee(\Sigma_{t}) = -\frac{\eps}{2}\int_{\Sigma_{t}} \left(\left(\frac{\partial u^{+}}{\partial \nu} \right)^{2} - \left(\frac{\partial u^{-}}{\partial \nu} \right)^{2} \right) \langle X, \nu \rangle d \mathcal{H}^{n - 1},
    \end{align}

    and

    \begin{align} \label{secvar}
        \frac{d^{2}}{dt^{2}} \bee(\Sigma_{t}) &= -\eps \int_{\Sigma} \left( \frac{\partial u^{+}}{\partial \nu} \frac{\partial \dot{u}^{+}}{\partial \nu} - \frac{\partial u^{-}}{\partial \nu} \frac{\partial \dot{u}^{-}}{\partial \nu}\right) \langle X, \nu \rangle \\
        &\quad\quad + \left(\frac{\partial u^{+}}{\partial \nu}\frac{\partial^{2} u^{+}}{\partial \nu^{2}} - \frac{\partial u^{-}}{\partial \nu}\frac{\partial^{2} u^{-}}{\partial \nu^{2}} \right) \langle X, \nu \rangle^{2} \nonumber \\
      & \quad \quad + \frac{1}{2}\left(\left(\frac{\partial u^{+}}{\partial \nu} \right)^{2} - \left(\frac{\partial u^{-}}{\partial \nu} \right)^{2} \right)(\langle Z, \nu \rangle + \langle X, \dot{\nu} \rangle - H_{\Sigma} \langle X, \nu \rangle^{2}) d \mathcal{H}^{n - 1} \nonumber 
    \end{align}
\end{proposition}

\begin{proof}
    By symmetry, it suffices to compute the variation on each side and then add together, so suppressing $+$ and $-$ signs and computing gives,

\begin{align*}
\frac{d}{dt} E_{\eps}(u_{t} ; M_{t}) &= \frac{d}{dt} \int_{M_{t}} \frac{\eps}{2} |\nabla u_{t}|^{2} + \frac{W(u)}{\eps} d \mathcal{H}^{n}\\
&= \int_{M_{t}} \eps \langle \nabla u_{t}, \nabla \dot{u}_{t} \rangle + \frac{W'(u_{t})\dot{u}_{t}}{\eps} d \mathcal{H}^{n}\\
& \quad \quad + \int_{\Sigma_{t}} \left( \frac{\eps}{2} |\nabla u_{t}|^{2} + \frac{W(u_{t})}{\eps}\right) \langle X_{t}, \nu_{t} \rangle d \mathcal{H}^{n - 1}\\
&= I_{1} + I_{2}.
\end{align*}

Integrating $I$ by parts and noting that $u_{t}$ vanishes along $\Sigma_{t}$ gives

\begin{align*}
    I_{1} &= \int_{M_{t}} \dot{u}_{t} \left(\frac{W'(u_{t})}{\eps} - \eps \Delta_{g}u_{t} \right) d \mathcal{H}^{n} + \int_{\Sigma_{t}} \eps \frac{\partial u_{t}}{\partial \nu_{t}} \dot{u}_{t} d \mathcal{H}^{n - 1}\\
    &= - \int_{\Sigma_{t}} \eps \left(\frac{\partial u_{t}}{\partial \nu_{t}} \right)^{2} \langle X_{t}, \nu_{t} \rangle d \mathcal{H}^{n - 1}.
\end{align*}

Now note for $I_{2}$ that since $u_{t}$ vanishes identically along $\Sigma_{t}$, it follows that $\nabla u_{t}$ is parallel to $\nu_{t}$. By the choice of the standard double well potential it similarly follows that $W(u_{t}) \equiv \frac{1}{4}$ along $\Sigma_{t}$. Putting that together gives

\begin{align*}
    I_{2} = \int_{\Sigma_{t}}  \left( \frac{\eps}{2} \left( \frac{\partial u_{t}}{\nu_{t}}\right)^{2} + \frac{1}{4\eps} \right) \langle X_{t}, \partial \nu_{t} \rangle d\mathcal{H}^{n - 1}.
\end{align*}

From which it follows that

\begin{align*}
\frac{d}{dt} E_{\eps}(u_{t}) = \int_{\Sigma_{t}} \left( -\frac{\eps}{2}  \left( \frac{\partial u_{t}}{\nu_{t}}\right)^{2} + \frac{1}{4\eps} \right) \langle X_{t}, \partial \nu_{t} \rangle d\mathcal{H}^{n - 1}.
\end{align*}

And so,

\begin{align*}
    \frac{d}{dt}\bee(\Sigma_{t}) &= \frac{d}{dt}E_{\eps}(u_{t}^{+}; M^{+}_{t}) + \frac{d}{dt} E_{\eps}(u^{-}_{t}; M^{-}_{t})\\
    &= -\frac{\eps}{2} \int_{\Sigma_{t}} \left(\left( \frac{\partial u^{+}_{t}}{\partial \nu_{t}}\right)^{2} - \left( \frac{\partial u^{-}_{t}}{\partial \nu_{t}}\right)^{2} \right) \langle X_{t}, \nu_{t} \rangle d \mathcal{H}^{n - 1}.
\end{align*}

To compute the second variation it suffices to compute

\begin{align*}
    I_{1} = \frac{d}{dt}\int_{\Sigma_{t}} \left( \frac{\partial u^{+}_{t}}{\partial \nu_{t}}\right)^{2} \langle X_{t}, \nu_{t} \rangle d \mathcal{H}^{n - 1}.
\end{align*}

For the purposes of simplifying notation we write $\Phi(x, t) = (\partial u^{+}_{t}/ \partial \nu_{t})^{2}(x, t)$. We pull back by the variation $F_{t}$ and compute.

\begin{align*}
    I_{1} &= \int_{\Sigma} \left( \frac{d}{dt} \Phi(F_{t}(x), t) \right) \langle X_{t}(F_{t}(x)), \nu_{t}(F_{t}(x))\rangle F_{t}^{\ast}d \mathcal{H}^{n-1}(x) \\
    & \quad \quad + \int_{\Sigma} \Phi(F_{t}(x), t) \left(\frac{d}{dt} \langle X_{t}(F_{t}(x)), \nu_{t}(F_{t}(x))\rangle\right) F_{t}^{\ast} d \mathcal{H}^{n-1}(x).\\
    & \quad \quad + \int_{\Sigma} \Phi(F_{t}(x), t) \langle X_{t}(F_{t}(x)), \nu_{t}(F_{t}(x)) \rangle \left( \frac{d}{dt} F_{t}^{\ast} d \mathcal{H}^{n - 1}(x)\right)\\
    &= I_{1} + I_{2} + I_{3}.
\end{align*}

The term in $I_{3}$ is the same as in the usual first variation of the area functional.

\begin{align*}
    I_{3} = -\int_{\Sigma_{t}} \Phi H_{\Sigma_{t}} \langle X_{t}, \nu_{t} \rangle^{2} d \mathcal{H}^{n - 1}.
\end{align*}

Dealing now with $I_{1}$, we compute

\begin{align*}
    \int_{\Sigma_{t}} \left(\frac{\partial \Phi}{\partial t} + \langle \nabla^{M} \Phi, X_{t} \rangle \right) \langle X_{t}, \nu_{t}\rangle d\mathcal{H}^{n-1}.
\end{align*}

Computing $I_{2}$ gives

\begin{align*}
    I_{2} &= \int_{\Sigma_{t}} \Phi \left(\langle \nabla_{\partial_{t}} X_{t}, \nu_{t} \rangle + \langle X_{t}, \nabla_{\partial_{t}}\nu_{t} \rangle \right) d\mathcal{H}^{n-1}\\
    &= \int_{\Sigma_{t}} \Phi\langle Z_{t}, \nu_{t} \rangle  d \mathcal{H}^{n-1},
\end{align*}

Putting all three parts together gives us

\begin{align*}
I_{1} + I_{2} + I_{3} &= \int_{\Sigma_{t}} \dot{\Phi}\langle X_{t}, \nu_{t} \rangle + (\Phi_{\nu_{t}} - \Phi H_{\Sigma_{t}}) \langle X_{t}, \nu_{t} \rangle^{2} +\\
&\quad \quad \Phi (\langle Z_{t}, \nu_{t} \rangle + \langle X_{t}, \nabla_{\partial_{t}}\nu_{t} \rangle ) d \mathcal{H}^{n - 1},
\end{align*}

which is precisely \ref{secvar}.
\end{proof}

\begin{remark}
    By selecting our variation $F_{t}$ as in \ref{goodVar}, we see that the terms in the second variation involving $Z$ and $\dot{\nu}$ vanish for all small $t$, and in this case we have simplified expressions for the first and second variation.
\end{remark}

\begin{corollary}\label{12Varf}
    Let $F_{t}$ be a variation such that $X_{t} = f\nu_{t}$ and $Z_{t} \equiv 0$ for $t$ sufficiently small. Then for such $t$ we have,

    \begin{align} \label{firvarF}
        \frac{d}{dt} \bee(\Sigma_{t}) = -\frac{\eps}{2}\int_{\Sigma_{t}} f\left(\left(\frac{\partial u^{+}}{\partial \nu} \right)^{2} - \left(\frac{\partial u^{-}}{\partial \nu} \right)^{2} \right)  d \mathcal{H}^{n - 1},
    \end{align}

    and

    \begin{align} \label{secvarf}
        \frac{d^{2}}{dt^{2}} \bee(\Sigma_{t}) &= -\eps \int_{\Sigma} f\left( \frac{\partial u^{+}}{\partial \nu} \frac{\partial \dot{u}^{+}}{\partial \nu} - \frac{\partial u^{-}}{\partial \nu} \frac{\partial \dot{u}^{-}}{\partial \nu}\right) \\
        & \quad \quad + f^{2}\left(\frac{\partial u^{+}}{\partial \nu}\frac{\partial^{2} u^{+}}{\partial \nu^{2}} - \frac{\partial u^{-}}{\partial \nu}\frac{\partial^{2} u^{-}}{\partial \nu^{2}} \right)  \nonumber \\
        & \quad \quad - H_{\Sigma}f^{2} \frac{1}{2}\left( \left(\frac{\partial u^{+}}{\partial \nu} \right)^{2} - \left(\frac{\partial u^{-}}{\partial \nu} \right)^{2} \right) d \mathcal{H}^{n - 1} \nonumber 
    \end{align}

\end{corollary}

\subsection{Estimate for the second variation.}

In this section, we use the asymptotic formulae \ref{uExp}, \ref{dotuExp} for $u^{\pm}$ and $\dot{u}^{\pm}$ to expand the second variation formula \ref{secvar}. More precisely we prove the following asymptotic formula for the second variation.

\begin{theorem}\label{secVarEstTheorem}
    Let $\Sigma_{t}$ be a $1$-parameter family of hypersurfaces as in Corollary \ref{12Varf}. Then 

    \begin{align}\label{2VarEst}
        \frac{d^{2}}{dt^{2}}\bee(\Sigma_{t}) &= 2\sigma_{0}\int_{\Sigma_{t}} |\nabla^{\Sigma_{t}} f|^{2} + (\textnormal{Ric}(\nu, \nu) + |A_{\Sigma_{t}}|^{2} - H_{\Sigma_{t}}^{2})f^{2} d \mathcal{H}^{n-1} + E(f)
    \end{align}

    Where $E(f) \leq C(\Sigma)\eps \|f\|_{W^{1,2}(\Sigma)}$.
\end{theorem}

\begin{proof}
    The proof is identical to the proof of Theorem 2.6 in \cite{marxkuo2023geometric} except in a few respects. The terms contributed by the variation of the volume form of $\Sigma_{t}$ and the terms containing two normal derivatives of the Dirichlet phase transitions do not appear in the calculation of the second variation at a critical point, so we deal with them here. The term containing a normal derivative of $\dot{u}$ is treated in exactly the same way, so we refer to the appendix of \cite{marxkuo2023geometric} for the proof.  We work in Fermi coordinates $(s, z)$ adapted to $\Sigma_{t}$. Recall,

    \begin{align}
        u_{\nu}^{+}(s, 0) &= \frac{1}{\eps\sqrt{2}} - \frac{2}{3} H_{\Sigma} + O(\eps), \label{unu}\\
        u_{\nu\nu}^{+}(s, 0) &= \frac{H_{\Sigma}}{\eps \sqrt{2}} - \frac{2}{3} H_{\Sigma}^{2} + O(\eps), \label{ununu}\\
        \dot{u}_{\nu}^{+}(s, 0) &= \phi^{+}_{z}(s, 0). \label{dotunu}
    \end{align}

    substituting in \ref{unu}, \ref{ununu} and \ref{dotunu} into the first two terms in \ref{secvarf} gives the following integral terms. 

    \begin{align}
        I_{1} \doteq - \eps \int_{\Sigma_{t}} f\left(\frac{1}{\eps\sqrt{2}} + O(1) \right)(\partial_{z}\phi^{+} - \partial_{z} \phi^{-})d\mathcal{H}^{n-1},
    \end{align}

    \begin{align}
        I_{2} &\doteq - \eps \int_{\Sigma_{t}} f^{2}\left(\frac{1}{\eps\sqrt{2}} -\frac{2}{3}H_{\Sigma_{t}} + O(\eps)\right)\left(\frac{H_{\Sigma_{t}}}{\eps\sqrt{2}} - \frac{2}{3}H_{\Sigma_{t}}^{2} + O(\eps) \right)\\
        &\quad \quad + f^{2}\left(-\frac{1}{\eps\sqrt{2}} + \frac{2}{3}H_{\Sigma_{t}} + O(\eps)\right)\left(\frac{H_{\Sigma_{t}}}{\eps\sqrt{2}} - \frac{2}{3}H_{\Sigma_{t}}^{2} + O(\eps) \right) d\mathcal{H}^{n-1}\nonumber
    \end{align}

    The third term contributes,

    \begin{align*}
        \frac{\eps}{2}\int_{\Sigma_{t}}H_{\Sigma_{t}}f^{2} \frac{1}{2}\left( \left(\frac{\partial u^{+}}{\partial \nu} \right)^{2} - \left(\frac{\partial u^{-}}{\partial \nu} \right)^{2} \right) d \mathcal{H}^{n - 1} = 2\sigma_{0}\int_{\Sigma_{t}}H_{\Sigma_{t}}^{2} f^{2} + C(\Sigma)\eps \|f\|_{L^{\infty}}^{2}
    \end{align*}

    by \ref{firstVarest0}. 

    The second summand in the integral $I_{2}$ appears with a reversed sign owing to the fact that the expansion of $u^{-}$ occurs with respect to the opposite normal vector as in Theorem \ref{uExp}. It follows that this term contributes only an error term which is $O(\eps)$ as $\eps \to 0$. Finally, to deal with the term $I_{1}$, we follow the proof of Theorem $(2.6)$ in \cite{marxkuo2023geometric}, only stating what the difference is in the current calculation.

    In the Fermi coordinates $(s, z)$, the Laplacian $\Delta_{g} = \partial_{z}^{2} - H_{z}\partial_{z} + \Delta_{z}$, where $H_{z}$ and $\Delta_{z}$ represent the mean curvature of the parallel hypersurface $\{p\in M|\text{d}(p, \Sigma_{t}) = z$, $\text{d}$ the signed distance function, and the Laplacian along the parallel hypersurface of distance $z$ respectively. Using the decomposition \ref{dotuExp}, it follows that $\ell_{\eps, u\pm}(\phi^{\pm}) = - \ell_{\eps, u^{\pm}}(h \overline{\het}_{\eps}')$, where $h = - f(s) u_{\nu}^{\pm}(s, 0)$. Expanding this using Fermi coordinates gives

    \begin{align*}
    \ell_{\eps, u} \phi = - \eps^{2}\Delta_{z}(h)\overline{\het}_{\eps}'+h \eps H_{z}  + (W''(u) - W''(\overline{\het}_{\eps}) + E)h\overline{\het}_{\eps}'
    \end{align*}

    Substituting this into the expression for each term of $I_{1}$,

    \begin{align}\label{expansion}
    \int_{\Sigma_{t}}f\phi^{+}_{z} &= - \frac{\sqrt{2}}{\eps^{2}}\int_{\Sigma_{t}}f \int_{0}^{-2\ell \log \eps} \bigg( -\eps^{2}\Delta_{z}(h)\overline{\het}_{\eps}^{'2} + \eps H_{z} h \overline{\het}_{\eps}^{''} \\
    &\quad \quad + (W''(u) - W''(\overline{\het}_{\eps}) + E) h \overline{\het}^{'2}_{\eps} \nonumber \\
    &\quad \quad - \sqrt{2}(\eps^{2}\Delta_{z}(\phi^{+}) - H_{z}\eps^{2}\phi^{+}_{z}\nonumber \\
    &\quad\quad - (W''(u) - W''(\overline{\het}_{\eps}) + E)\phi^{+})\overline{\het}_{\eps}'\bigg) \sqrt{\det g(s, z)}dzds \nonumber
    \end{align}

    That this term is equal to 

    \[
    \int_{\Sigma_{t}} |\nabla f|^{2} - (\textnormal{Ric}(\nu, \nu) + |A_{\Sigma_{t}}|^{2})f^{2} d\mathcal{H}^{n-1} + E(f)
    \]

    where $E(f) \leq C\eps^{\frac{1}{2}}\|f\|_{W^{1,2}}$ is the subject of \cite{marxkuo2023geometric}, section 11.4, the only difference in our situation, is that $H_{z}$ is not necessarily small in $\eps$, and so cannot be absorbed into the remainder, instead we note that the terms which appear with $H_{z}$ in the integral \ref{expansion} appear with opposite contributions from $\phi^{+}_{z}$ and $\phi^{-}_{z}$ respectively. 
\end{proof}

\section{The Strong min-max property}\label{minmaxSec}

We are ready to prove \ref{min max}.

\begin{proof}[Proof of Theorem \ref{min max}]
    We proceed by contradiction. Suppose there exists a $\delta > 0$ so that we can find for each $i$ a $k$-parameter family of surfaces $\{\Sigma_{p}^{i} | p \in B^{k}\}$ in $\mathcal{U}$ with

    \[
    \sup_{p \in B_{k}} \bei(\Sigma_{p}^{i}) \leq \bei(\Sigma) - \delta.
    \]

    In particular, for each $i$ we can find a sequence of functions $f_{i}$ so that each normal graph $\Gamma(f_{i})$ satisfies

    \[
    \bei(\Gamma(f_{i})) \leq \bei(\Sigma) - \frac{\delta}{2}.
    \]

    Taylor expanding $\bei$ around $f_{i}$ and applying Corollary 2.1 in \cite{marxkuo2023geometric} gives

    \[
    \bei(\Gamma(f_{i})) = \bei(\Sigma) + R_{\eps_{i}} + T_{\eps_{i}}(f_{i}),
    \]

    where $T_{\eps_{i}}$ is the remainder term in the first order Taylor expansion and,

    \[
    R_{\eps_{i}} = -\int_{\Sigma} H_{\Sigma} f + E_{i}(f),
    \]

    and $|E_{i}(f)| \leq C\eps_{i}$ where $C$ is a constant depending only on $\Sigma$. Since $\Sigma$ is a minimal surface it follows that $R_{\eps_{i}} = O(\eps_{i})$ as $i \to \infty$. It follows from Taylor's theorem that

    \[
    T_{i}(f_{i}) = \frac{1}{2} \frac{d^{2}}{dt^{2}}\bigg\vert_{t = \xi_{i}} \bei(\Gamma(\xi_{i} f_{i})),
    \]

    where $\xi_{i} \in (0, 1)$. However, from Theorem \ref{secVarEstTheorem}, this is 

    \[
    \sigma_{0}\frac{d^{2}}{dt^{2}}\bigg\vert_{t = \xi_{i}} \mathcal{A}(\Gamma(\xi_{i}f_{i}))  + \tilde{R}_{\eps_{i}}
    \]

    where $R_{\eps_{i}} \to 0$ as $i \to \infty$. In particular, it follows that $T_{\eps_{i}}(f_{i}) - 2\sigma_{0}T(f_{i}) \to 0$ as $i \to \infty$, where $T(f_{i})$ is the Taylor remainder term when expanding the area functional about $f_{i}$. Therefore it follows that, for $i$ sufficiently large

    \[
    |\bei(\Gamma(f_{i})) - \mathcal{A}(\Gamma(f_{i}))| \leq \frac{\delta}{4}.
    \]

    So for any such $i$ we have

    \[
    \sigma_{0}\sup_{p \in B^{k}} \mathcal{A}(\Sigma_{p}^{i}) \leq \frac{\delta}{4} + \sup_{p \in B^{k}} \bei(\Sigma_{p}^{i}) \leq \bei(\Sigma) - \frac{\delta}{4}.
    \]

    However, $\bei(\Sigma) \to \mathcal{A}(\Sigma)$, so this implies that

    \[
    \sup_{p \in B^{k}} \mathcal{A}(\Sigma_{p}^{i}) < \mathcal{A}(\Sigma),
    \]

    which violates \ref{WhiteTheorem}.

\end{proof}

\section{A weak Palais-Smale condition}\label{PSSec}

Let $\mathcal{S}^{k, \alpha}$ denote the space of all $2$-sided, closed, $C^{k, \alpha}$-hypersurfaces of $M$ which are homologous to $\Sigma$ (and moreover which are therefore boundaries). We wish to show the following Palais-Smale type compactness condition holds.

\begin{proposition} \label{PSc}
\noindent Let $\{\Sigma_{n}\}$ be a sequence of hypersurfaces in $\mathcal{S}^{k, \alpha}$ and suppose $d > 0$. Then if \bigskip

\begin{enumerate}
    \item $\delta\bee(\Sigma_{n}) \to 0$,
    \item $\bee(\Sigma_{n}) \to d$,
    \item $\sup_{n} \textnormal{Area}(\Sigma_{n}) < \infty$
\end{enumerate}
\bigskip

then, up to taking a subsequence, there exists a Caccioppoli set $M^{+}$ for which $M^{+}_{n} \to M^{+}$ and $\Sigma_{\infty} \doteq \partial M^{+}$, and moreover there is a smooth phase transition $u: M \to \bf{R}$ which vanishes on $\Sigma_{\infty}$
\end{proposition}

\begin{remark}
    The boundary $\Sigma_{\infty}$ need not be smooth, as there can be points along $\Sigma_{\infty}$ where $\nabla u = 0$. However, the phase transition $u$ will be smooth everywhere by standard elliptic regularity.
\end{remark}

\begin{proof}[Proof of Proposition \ref{PSc}]
Suppose $u_{n}$ denotes the broken phase transition associated to the hypersurface $\Sigma_{n}$. Then it is easy to see that $u_{n}$ satisfies

\begin{align} 
    E_{\varepsilon}'(u_{n}) \to 0 \\
    E_{\varepsilon}(u_{n}) \to d \label{nonzero energy}
\end{align}

Therefore, by the classical Palais-Smale condition, there is a function $u_{\infty}$ with $u_{n} \xrightharpoonup[W^{1,2}]{} u_{\infty}$. It follows that $u_{\infty}$ is a weak solution to the Allen-Cahn equation in $W^{1, 2}(M)$, and therefore by elliptic regularity it is a smooth phase transition. By \ref{nonzero energy}, $M^{+} \doteq \{u_{\infty} > 0\}$ and $M^{-} \doteq \{u_{\infty} < 0\}$ are nonempty. We note that on $M^{+}$, $u_{\infty}$ is a positive Dirichlet solution to the problem

\[
\begin{cases}
\Delta_{g} u_{\infty} = \frac{W'(u)}{\varepsilon^{2}}, &\text{in } M^{+},\\
u_{\infty} \equiv 0, &\text{on }\partial M^{+}.
\end{cases}
\]

By the weak convergence of $u_{n}$ to $u_{\infty}$ in $W^{1, 2}$, it follows from Rellich's theorem that $u_{n} \to u_{\infty}$ strongly in $L^{2}$. Moreover, by  compactness of $M$, we may pass to a further subsequence with $u_{n} \to u_{\infty}$ in $L^{1}$. Note that the measure of $\partial M^{+}$ is $0$ by the results of \cite{caffarelli1985partial}. It suffices to show the following claim.

\begin{claim}\label{subclaim1}
    $M_{n}^{+} \to M^{+}$.
\end{claim} 

\begin{proof}[Proof of Claim \ref{subclaim1}]
    Suppose it is not the case. Then, possibly up to taking a subsequence, we may find a $\delta > 0$ so that $|M^{+}_{n} \Delta M^{+}| > \delta$. However, the area bound on $\Sigma_{n}$ implies that $\{M_{n}^{+}\}$ is of uniformly bounded variation, so $\text{BV}$-compactness (\cite{maggi2012sets}, Theorem 12.26) then implies that there is a measurable set $M_{\infty}^{+}$ so that

    \[
    M_{n}^{+} \to M_{\infty}^{+}.
    \]

    \noindent Therefore it follows that

    \[
    |M_{\infty}^{+} \Delta M^{+}| \geq \delta.
    \]

    \noindent So moreover there is, for some $k_{0} < 0$, a set $A \subset M_{\infty}^{+}$ of positive measure so that

    \[
    u_{\infty}\big\vert_{A} \leq k_{0}.
    \]

    \noindent However, $u_{n} \geq 0$ along $A$ by construction, and $u_{n}\big\vert_{A} \xrightarrow[L^{1}]{} u_{\infty}\big\vert_{A}$, which is a contradiction.
\end{proof}

The proof is now complete.
\end{proof}

\section{A mountain pass theorem}\label{MPSec}

\subsection{A Deformation Lemma}\label{DefLemmaSec}

 In order to prove an appropriate mountain-pass type theorem, we need to prove some variant of the usual deformation lemma \cite{rabinowitz2011variational}. 

\bigskip

Let $\Sigma \subset M$ denote a fixed, 2--sided, separating $C^{\infty}$ hypersurface for which $\bee$ is well-defined. For $f \in C^{k, \alpha}(\Sigma)$ we define $\bee(f) \doteq \bee(\Gamma(f))$, provided it is defined. Note that there is an $r > 0$ so that $\bee : B^{k, \alpha}_{r}(0) \to \mathbf{R}$ is well-defined, where $B^{k, \alpha}_{r}(f)$ denotes the ball of radius $r$ in $C^{k, \alpha}(\Sigma)$ about the function $f$. Let $c, \rho \in \mathbf{R}$ with $\rho < r$ . Let $U \subset B_{r}^{k, \alpha}(0)$. We define the following sets.

\begin{align}\label{Spaces}
    &\mathcal{O}_{U} = \{\Omega \in \mathcal{C}(M)| \exists \Omega_{n} \in \mathcal{C}(M), \partial \Omega_{n} = \Gamma(f), f \in U, |\Omega \Delta \Omega_{n}| \to 0\}.\\
    &A_{c}^{\eps} \doteq \{f \in C^{k, \alpha}(\Sigma) \vert \bee(f) \leq c\},\\
    &K_{c, U}^{\eps} \doteq \{\Omega \in \mathcal{O}_{U} \vert \partial \Omega = u^{-1}(0), E_{\eps}(u) = c, \delta E_{\eps}(u) = 0 \},\\
    &K_{U}^{\eps} \doteq \{\Omega \in \mathcal{O}_{U} \vert \partial \Omega = u^{-1}(0), \delta E_{\eps}(u) = 0 \}
    \end{align}

Let $\delta > 0$ and suppose $U \subset B^{k, \alpha}_{r}(0)$ is an admissible neighborhood of $0$ and that $B_{\delta}^{k, \alpha}(U), B_{2\delta}^{k, \alpha}(U) \subset B^{k, \alpha}_{r}(0)$ are also admissible.

 We wish to show that if $K_{c, U}^{\eps}$ is empty, we can always deform $A_{c + \eta}$ into $A_{c - \eta}$ for some appropriate $\eta > 0$ without running into a critical point.

\begin{lemma}[Localized deformation] \label{DefLemma}
Suppose $c > 0$ is such that $K^{\eps}_{c, U_{2\delta}} = \emptyset$. There are constants $0 < \bar{\eta} < \eta$ and a homeomorphism $F \in C([0, 1] \times C^{k, \alpha}(\Sigma); C^{k, \alpha}(\Sigma))$ so that

\begin{enumerate}
    \item $F(0, u) = u,$
    \item $F(1, u) = u$ for all $u \not \in \bee^{-1}[c - \eta, c + \eta] \cap U_{2\delta}$,
    \item $F(1, A_{c + \bar{\eta}} \cap U) \subset A_{c - \bar{\eta}} \cap U_{\delta}$.
    \item $F(t, u)$ is a homeomorphism for all $t \in [0, 1]$.
\end{enumerate}
\end{lemma}

\begin{proof}

\noindent First we claim the following.

    \begin{claim}\label{subclaim2}
        There are constants $0 < \mu, \eta < 1$ so that 

        \[
        \|\delta\bee(\Gamma(u))\| \geq \mu
        \]

        whenever $u \in (A_{c + \eta} \backslash A_{c - \eta}) \cap U_{2\delta}$.
    \end{claim}

    \begin{proof}[Proof of claim \ref{subclaim2}]
        \noindent Suppose this is not the case. Then there are sequences $\mu_{k}, \eta_{k} \to 0$ as $k \to \infty$ and $u_{k} \in (A_{c + \eta_{k}} \backslash A_{c - \eta_{k}}) \cap U_{2\delta}$ with

        \begin{align}
            &\|\delta\bee(\Gamma(u_{k}))\| \leq \mu_{k},\\
            &\bee(\Gamma(u_{k})) \to c.
        \end{align}

        \noindent The condition \ref{PSc} then implies the existence of $\Omega \in K_{\eps, U_{2\delta}}^{c}$, which contradicts our assumption that $K^{\eps}_{c, U_{2\delta}}  = \emptyset$.
    \end{proof}

    \noindent Now set $\bar{\eta}$ so that

    \[
    0 < \bar{\eta} < \eta, 0 < \bar{\eta} < \frac{\mu^{2}}{2}.
    \]

    \noindent We also set

    \begin{align*}
        &A \doteq \{u \in C^{k, \alpha}(\Sigma) \vert c - \bar{\eta} \leq \bee(u) \leq c + \bar{\eta}\} \cap U_{\delta},\\
        &B \doteq C^{k, \alpha}(\Sigma) \backslash (\{u \in C^{k, \alpha}(\Sigma) \vert c - \bar{\eta} \leq \bee(u) \leq c + \bar{\eta}\} \cap U_{2\delta}).
    \end{align*}

    \noindent We define functions $d: C^{k, \alpha}(\Sigma) \to \mathbf{R}$, $h: \mathbf{R}^{+} \to \mathbf{R}$ by

    \begin{align*}
        &d(u) = \text{dist}(u, B)(\text{dist}(u, A) + \text{dist}(u, B))^{-1},\\
        &h(t) = \begin{cases}
            1, & 0 \leq t \leq 1,\\
            1/t, & t \geq 1.
        \end{cases}
    \end{align*}

    \noindent The function $d$ satisfies $0 \leq d \leq 1$, $d\vert_{B} \equiv 0, d\vert_{A} \equiv 1$ and is locally Lipschitz.

Now we may define a vector field $G:C^{k,\alpha}(\Sigma) \to C^{k, \alpha}(\Sigma)$ by

    \[
    G(u) \doteq -d(u) h(\|\bee'(u)\|) V(u).
    \]
    
    Here $V$ is a pseudogradient field for $\bee$ defined on $B^{r}_{k,\alpha}(0)$ (for the definition of a pseudogradient vector field see \cite{rabinowitz2011variational}, and note that Lemma $1.6$ of \cite{rabinowitz2011variational}, shows that such a $V$ exists, is locally Lipschitz and satisfies $\|V\| \leq 2 \|\bee'\|$ and $\langle \bee', V \rangle \geq \|\bee'\|^{2}$). Note that $G$ is well-defined, as $d$ vanishes outside of $B^{r}_{k,\alpha}(0)$. We claim that it is bounded and locally Lipschitz, and therefore by the general existence and uniqueness theory for ODEs on Banach spaces (see \cite{lang1985differential}, Chapter IV), a solution exists for some maximal time $t^{+} \geq 0$ to the equation

    \begin{align}
        \begin{cases}
            \dot{w}(t) = G(w(t)),\\
            w(0) = u.
        \end{cases}
    \end{align}

    In fact, from boundedness, it then follows that a solution exists for all time. To see this suppose $t^{+} < \infty$ and let $t_{n} \to t^{+}$ from below. Then,

    \[
    \|w(t_{n +1}) - w(t_{n})\| = \left\| \int_{t_{n}}^{t_{n + 1}} G(w(t)) dt \right\| \leq \|G\||t_{n + 1} - t_{n}|.
    \]

    It follows that $w(t_{n}, u)$ converges to some value $\overline{w}$. By local existence and uniqueness we may solve the ODE with initial value $\overline{w}$ for some small time, which now contradicts maximality of $t^{+}$.

    We set $F(t, u) \doteq w(\delta t, u)$. $(1)$, $(2)$ and $(4)$ now follow from the definition of the vector field $G$ and the initial value problem above.

    To see that $G$ is bounded, recall \ref{firstVarest0}, which says that we may rewrite $G(u)$ as
    
    \[
    \|G(u)\| = \left\|4\sigma_{0}d(u) h(\|\bee'(u)\|)\left( \int_{\Sigma} H_{\Gamma(u)}  + E(u)\right)\right\|
    \]

    where $|E(u)| \leq K(\eps) \|u\|_{C^{k, \alpha}(\Sigma)}$. Moreover G vanishes outside of $B_{k, \alpha}^{r}(0),$  $|d|, |h| \leq 1$, and the mean curvature of $\Gamma(u)$ is controlled by the $C^{k, \alpha}$ norm of $u$, so $G$ is bounded.

    \noindent To show $(3)$ suppose that $u \in A_{c + \bar{\eta}} \cap U$. Note that for $t > 0$,

    \begin{align*}
        \|w(t, u) - u\|_{C^{k, \alpha}(\Sigma)} &= \left\|\int_{0}^{t} G( w(\tau, u)) d\tau \right\|_{C^{k, \alpha}(\Sigma)}\\
        &\leq \int_{0}^{t} \|G( w(\tau, u))\|_{C^{k, \alpha}(\Sigma)} d\tau\\
        &\leq t.
    \end{align*}

    It follows for $0 < t < \delta$ that $w(t, u) \in U_{\delta}$ for any $u \in U$.

    If there were some $t$ with $0 \leq t < \delta$ with $\bee(w(t, u)) < c - \bar{\eta}$, then $\bee(w(\delta, u)) < c - \bar{\eta}$, as $\bee(w(t, u))$ is decreasing in $t$, and therefore we would have $w( \delta, u) \in A_{c - \bar{\eta}} \cap U_{\delta}$. Otherwise, we may would have that for all $t \in [0, \delta]$, $w(t, u) \in A$, as

    \[
    c - \bar{\eta} \leq \bee(w(t, u)) \leq \bee(u) \leq c + \bar{\eta}.
    \]

    It follows from the fact that $d \equiv 1$ on $A$ that

    \[
    \frac{d}{dt} \bee(w(t, u)) \leq - h(\|\bee'(w(t, u))\|) \|\bee(w(t, u))\|^{2}.
    \]

    \noindent If $\|\bee'(w(t, u))\| \geq 1$, 

    \[
    \frac{d}{dt} \bee(w(t, u)) = -\|\bee(w(t, u))\| \leq -\mu \leq - \mu^{2}.
    \]

    \noindent Otherwise if $\|\bee'(w(t, u))\| \leq 1,$ then by definition of $h$ we have

    \[
    \frac{d}{dt} \bee(w(t, u)) \leq - \mu^{2}.
    \]

    \noindent Either way we find 

    \[
    \bee(w(\delta, u)) \leq \bee(u) - \mu^{2} \leq c - \bar{\eta},
    \]

    which completes the proof.
\end{proof}

In an admissible neighborhood of a nondegenerate minimal hypersurface, we may derive more quantitative information about the deformation. First suppose $\tilde{\eps} > 0$ is fixed.

\begin{lemma}\label{QuantDef}
    Let $\Sigma$ be a $2$-sided, separating, nondegenerate minimal hypersurface. There is a neighborhood $U_{0}$ of $0$ in $C^{k, \alpha}(\Sigma)$ and $\delta_{0} > 0$ so that $U_{0}, U_{0, 2\delta_{0}}$ are $\tilde{\eps}$-admissible neighborhoods, so that for every $U \subset U_{0}$ where $U$ is an $\tilde{\eps}$-admissible neighborhood of $0$, and $\delta < \delta_{0}$, exactly one of the following holds

    \begin{enumerate}
        \item There is $\eps_{0}$ with $0 < \eps_{0} < \tilde{\eps}$ so that for every $\eps < \eps_{0}$, $K_{U_{\delta}}^{\eps} \neq \emptyset$ \label{case1}
        \item There are constants $\mu_{0}, \eps_{0}$ depending on $U_{0}, \delta_{0}, \tilde{\eps}, \Sigma$, so that for every $\eps < \eps_{0}$ and $u \in U_{\delta} \backslash U$

    \[
    \|\delta \bee(\Gamma(u))\| \geq \mu_{0}.
    \]
    \end{enumerate}

\end{lemma}

\begin{proof}
    Suppose (\ref{case1}) is not the case. Then if $(2)$ were not true, it would follow that there are neighborhoods $U_{i} \subset {U_{0}}$ with $\bigcap U_{i} = \{0\}$, a sequence $\delta_{i} \to 0$ as $i \to \infty$, sequences $\eps_{i, j}, \mu_{i,j} \to 0$ as $j \to \infty$, and functions $u_{i, j} \in U_{i, \delta_{i}}\backslash U_{i}$ with

    \begin{align}\label{QuantDef1}
    \|\delta \mathcal{B}_{\eps_{i, j}}(\Gamma(u_{i, j}))\| \leq \mu_{i, j}.
    \end{align}

    Write $\Gamma_{i, j}$ for $\Gamma(u_{i, j})$. Note that applying the first variation formula \ref{firstVarest0} to \ref{QuantDef1} tells us that the mean curvatures $H_{\Gamma_{i, j}}$ satisfy

    \begin{align}\label{QuantDef2}
    H_{\Gamma_{i, j}} = o(1) \text{ as } j \to \infty.
    \end{align}

    Since the sequence $u_{i, j}$ is uniformly bounded in $C^{k, \alpha}(\Sigma)$, we may extract a convergent subsequence by the Arzel\`a-Ascoli theorem converging in $C^{k, \beta}(\Sigma)$ for any $\beta < \alpha$ to a nonzero function; call it $\tilde{u}_{i}$. As $k > 2$, the curvatures $H_{\Gamma_{i, j}}$ also converge $H_{\tilde{\Gamma}_{i}}$ (here $\tilde{\Gamma}_{i} \doteq \Gamma(\tilde{u}_{i})$). It follows from \ref{QuantDef2} that the graph $\tilde{\Gamma}_{i}$ of $\tilde{u}_{i}$ is a minimal hypersurface. 

    We have thus constructed a sequence $\tilde{\Gamma}_{i}$ of minimal hypersurfaces, and $\tilde{u}_{i} \to 0$ by construction. It follows that $\tilde{\Gamma}_{i} \to \Sigma$. A standard argument (see, for example, Lemma 4.1 in \cite{sharp2017compactness}) now gives us the existence of a nontrivial Jacobi field on $\Sigma$, contradicting nondegeneracy.
\end{proof}

\begin{remark}
    The quantitative lemma \ref{QuantDef} allows us to say that, when near a nondegenerate minimal hypersurface, provided that we do not run into any phase transitions when applying the deformation lemma \ref{DefLemma}, there is a uniform (in $\eps)$ amount by which we can decrease $\bee$, so long as $\eps$ is sufficiently small.
\end{remark}

\subsection{$1$-parameter mountain pass construction.}\label{MPSubsec1}

With the deformation lemma in hand, we are now ready to prove a mountain pass type result. For simplicity we first present the $1$-parameter version. Let $K_{c}, A_{c}$ be as defined previously

\begin{theorem}[$1$-parameter Mountain Pass Theorem]\label{MPTheorem}
Suppose $\Sigma$ is a $2$-sided, separating, nondegenerate minimal hypersurface and suppose $U, U_{\delta}, U_{2\delta}$ are admissible neighborhoods of $\Sigma$, where $\delta > 0$. There exists $\eps_{0}$ so that for all $\eps < \eps_{0}$ the following holds. Let $u_{0}, u_{1} \in U$ be distinct. Define $\mathcal{P}, c, d$ by

\begin{align}
    &\mathcal{P} \doteq \{p \in C([0, 1], U) \vert p(0) = u_{0}, p(1) = u_{1}\},\\
    &\mathcal{P}_{\delta} \doteq \{p \in C([0, 1], U_{\delta}) \vert p(0) = u_{0}, p(1) = u_{1}\},\\
    &c_{\eps} \doteq \max \{ \bee(p(0)), \bee(p(1))\},\\
    &d_{\eps} \doteq \inf_{p \in \mathcal{P}} \max_{t \in [0, 1]} \bee(p(t))\\
\end{align}

\noindent Then, if $\lim_{\eps \to 0} c_{\eps} < \text{Area}(\Sigma)$, $K_{d_{\eps}}^{\eps} \neq \emptyset$.
\end{theorem}

\begin{proof}
If it were not the case, then we could find infinitely many values of $\eps$ with $\eps \to 0$ so that $K^{\eps}_{d_{\eps}}$ is empty. We set

\begin{align}
    d_{\eps}^{\delta} \doteq \inf_{p \in P_{\delta}} \max_{t \in [0, 1]} \bee(p(t)),
\end{align}

and we note that possibly $d_{\eps}^{\delta} < d_{\eps}$. However, from the continuity in $\eps$ of the min-max values \ref{min max} it follows that $d_{\eps} - d_{\eps}^{\delta} = o(1)$ as $\eps \to 0$, and $d_{\eps} \to \text{Area}(\Sigma)$. Therefore it follows from \ref{QuantDef} and \ref{DefLemma} that there is an $\eps_{0}$ so that for all $\eps < \eps_{0}$ we can find positive constants $\bar{\eta}, \eta$ satisfying $d_{\eps} - d_{\eps}^{\delta} < \bar{\eta} < \eta < d_{\eps} - c_{\eps}$ and homeomorphisms $F^{\eps}: C^{k, \alpha}(\Sigma) \to C^{k,\alpha}(\Sigma)$ satisfying

\begin{align}
    &F^{\eps}(A_{d_{\eps} + \bar{\eta}}^{\eps} \cap U) \subset A_{d_{\eps} - \bar{\eta}}^{\eps} \cap U_{\delta},\\
    &F^{\eps}(u) = u & \text{ if } u \not\in \bee^{-1}[d_{\eps} - \eta, d_{\eps} + \eta] \cap U_{2 \delta}.
\end{align}

\noindent Take some path $a \in \mathcal{P}$ satisfying

\[
\max_{t \in [0, 1]} \bee(a(t)) \leq d_{\eps} + \bar{\eta}.
\]

\noindent Then $\bar{a} \doteq F^{\eps}(a) \in \mathcal{P}_{\delta}$. The deformation lemma now implies that

\[
\max_{t \in [0, 1]} \bee(\bar{a}(t)) \leq d_{\eps} - \bar{\eta} < d_{\eps}^{\delta},
\]

a contradiction. 
\end{proof}

\subsection{$l$-paramter mountain pass construction}

\begin{theorem}[$l$-parameter Mountain Pass Theorem]\label{MPTheoremk}
Let $\Sigma$ is a $2$-sided, separating, nondegenerate minimal hypersurface and suppose $U, U_{\delta}, U_{2\delta}$ are admissible neighborhoods of $\Sigma$, where $\delta > 0$. There exists $\eps_{0}$ so that for all $\eps < \eps_{0}$ the following holds. Consider an $l$-dimensional subspace $Z \subset C^{k,\alpha}(\Sigma)$ and $\rho > 0$ with $B_{\rho}(0) \subset U$ and set $Q \doteq \partial B_{\rho}(0) \cap Z$. Let $\gamma_{0} \in C^{0}(\partial B^{l}, Q)$ be fixed. Define $\mathcal{P}, P_{\delta}, c_{\eps}, d_{\eps}$ by

\begin{align}
    &\mathcal{P} \doteq \{p \in C(B^{l}, U) \vert p\vert_{\partial B^{l}} = \gamma_{0},\\
    &\mathcal{P}_{\delta} \doteq \{p \in C([0, 1], U_{\delta}) \vert p\vert_{\partial B^{l}} = \gamma_{0}\},\\
    &c_{\eps} \doteq \max_{v \in \partial B^{l}} \{\bee(p(v)) \},\\
    &d_{\eps} \doteq \inf_{p \in \mathcal{P}} \max_{v \in \partial B^{l}} \bee(p(v))\\
\end{align}

\noindent Then, if $\lim_{\eps \to 0} c_{\eps} < \text{Area}(\Sigma)$, $K_{d_{\eps}}^{\eps} \neq \emptyset$.
\end{theorem}

\begin{proof}
The proof is identical to the case of $1$-parameter. If it were not the case, then we could find infinitely many values of $\eps$ with $\eps \to 0$ so that $K^{\eps}_{d_{\eps}}$ is empty. We set

\begin{align}
    d_{\eps}^{\delta} \doteq \inf_{p \in P_{\delta}} \max_{v \in \partial B^{l}} \bee(p(v)),
\end{align}

and we note that possible $d_{\eps}^{\delta} < d_{\eps}$. However, from the continuity in $\eps$ of the min-max values \ref{min max} it follows that $d_{\eps} - d_{\eps}^{\delta} = o(1)$ as $\eps \to 0$ and $d_{\eps} \to \text{Area}(\Sigma)$. Therefore it follows from \ref{QuantDef} and \ref{DefLemma} that there is an $\eps_{0}$ so that for all $\eps < \eps_{0}$ we can find positive constants $\bar{\eta}, \eta$ satisfying $d_{\eps} - d_{\eps}^{\delta} < \bar{\eta} < \eta < d_{\eps} - c_{\eps}$ and homeomorphisms $F^{\eps}: C^{k, \alpha}(\Sigma) \to C^{k,\alpha}(\Sigma)$ satisfying

\begin{align}
    &F^{\eps}(A_{d_{\eps} + \bar{\eta}}^{\eps} \cap U) \subset A_{d_{\eps} - \bar{\eta}}^{\eps} \cap U_{\delta},\\
    &F^{\eps}(u) = u & \text{ if } u \not\in \bee^{-1}[d_{\eps} - \eta, d_{\eps} + \eta] \cap U_{2 \delta}.
\end{align}

\noindent As in the proof of the theorem for $1$ parameter. some map $a \in \mathcal{P}$ satisfying

\[
\max_{v \in \partial B^{l}} \bee(a(v)) \leq d_{\eps} + \bar{\eta}.
\]

\noindent Then $\bar{a} \doteq F^{\eps}(a) \in \mathcal{P}_{\delta}$. The deformation lemma now implies that

\[
\max_{v \in \partial B^{l}} \bee(\bar{a}(v)) \leq d_{\eps} - \bar{\eta} < d_{\eps}^{\delta},
\]

a contradiction. 
\end{proof}

Suppose $\Sigma$ now is a $2$-sided, separating, nondegenerate minimal hypersurface of index $l$. In order to see how \ref{PacRit} follows from theorem \ref{MPTheoremk}, let $U$ be an admissible neighborhood of $\Sigma$ and set $Z$ to be the subspace of $C^{k, \alpha}(\Sigma)$ spanned by the first $l$-eigenfunctions of the stability operator of $\Sigma$. It follows immediately that, for $\rho > 0$ sufficiently small that $B_{\rho}(0) \cap Z \subset U$, that 

\[
\lim_{\eps \to 0} \max_{u \in \partial B_{\rho}(0) \cap Z} \bee(u) < \text{Area}(\Sigma),
\]

and so, applying theorem \ref{MPTheoremk} the existence of a smooth $\eps$-phase transition whose Allen--Cahn energy is $d_{\eps}$ and whose zero set is contained in the closure of $U_{2\delta}$. Applying the argument to a decreasing sequence of admissible neighborhoods of $\Sigma$ gives a sequence of such $\eps$-phase transitions with $\eps \to 0$. The result follows by the continuity of the min-max values \ref{min max}.
\appendix

\section{Smooth dependence on the domain.}\label{SmoothDependenceDomainssec}

Let $k > 2$ be a given integer, and suppose that $\Omega \subset (M, g)$ is a domain with $C^{k, \alpha}$ boundary (i.e., the inclusion map $i_{\Omega} \in C^{k, \alpha}(\Omega; M)$) which supports a positive Dirichlet phase transition vanishing on the boundary. Take a neighborhood $U \subset C^{k, \alpha}(\Omega; M)$ consisting only of $C^{k, \alpha}$ embeddings of $\Omega$. Suppose that $F \in U$ and let $u_{F}$ be the unique positive solution, afforded by \cite{brezis1986remarks}, to the problem
\[
\begin{cases}
\eps^2 \Delta_{g} u = W'(u) & \text{in } F(\Omega),\\
u \equiv 0 & \text{on } \partial F(\Omega).
\end{cases}
\]
In \cite{marxkuo2023geometric} it was shown that the dependence $F \mapsto u_{F}$ is smooth, provided the neihgborhood $U$ is taken sufficiently small. Let $F_{t}$ be a smooth one parameter family of diffeomorphisms with $F_{0} = \text{id}_{M},$ $\partial_{t}F_{t} \doteq X_{t}$ and $\nabla_{\partial_{t}}X_{t} = 0 $. Moreover, set $X \doteq X_{0}$ for the variation vector field. It follows that $t \mapsto u_{t} \doteq u_{F_{t}}$ is a smooth one parameter family of maps. It is a straightforward computation to show that the time-derivative $\dot{u}_{t}$ satisfies the boundary value problem:

\begin{align} \label{eqfordotu}
    \begin{cases}
        \Delta_{g}\dot{u}_{t} = W''(u_{t})\dot{u_{t}} \text{ in } \Omega_{t},\\
        \dot{u}_{t} = - \langle X_{t}, \nabla^{M} u_{t} \rangle \text{ on } \partial \Omega_{t},
    \end{cases}
\end{align}

\printbibliography
\end{document}